\documentclass[preprint]{elsarticle}
\usepackage[utf8]{inputenc}
\usepackage{epsfig}
\usepackage{amssymb,amsmath,amsthm} 
\usepackage{amsfonts}
\usepackage{algorithm,algpseudocode}
\usepackage{color}
\usepackage{epstopdf}
\usepackage{hyperref}
\usepackage{changes}
\usepackage{todonotes}

\newcommand{\trasp}{^{\mathsf{T}}}
\newcommand{\mnf}{\mathcal{M}_r}
\newcommand{\ee}{\mathrm{e}}
\newcommand{\dd}{\,\textrm{d}}

\newtheorem{lemma}{Lemma}

\begin{document}

\begin{frontmatter}
	
\title{Numerical low-rank approximation of \\ matrix differential equations}
\author[ibk,yachay]{Hermann~Mena}
\ead{mena@yachaytech.edu.ec}
\author[ibk]{Alexander~Ostermann}
\ead{alexander.ostermann@uibk.ac.at}
\author[ibk]{Lena-Maria~Pfurtscheller}
\ead{lena-maria.pfurtscheller@uibk.ac.at}
\author[ibk]{Chiara~Piazzola\corref{cor1}}
\ead{chiara.piazzola@uibk.ac.at}
\address[ibk]{Institut f\"ur Mathematik, Universit\"at Innsbruck, Austria}
\address[yachay]{Department of Mathematics, Yachay Tech, Urcuqu\'i, Ecuador}

\cortext[cor1]{Corresponding author}

\begin{abstract}
The efficient numerical integration of large-scale matrix differential equations is a topical problem in numerical analysis and of great importance in many applications. Standard numerical methods applied to such problems require an unduly amount of computing time and memory, in general. Based on a dynamical low-rank approximation of the solution, a new splitting integrator is proposed for a quite general class of stiff matrix differential equations. This class comprises differential Lyapunov and differential Riccati equations that arise from spatial discretizations of partial differential equations. The proposed integrator handles stiffness in an efficient way, and it preserves the symmetry and positive semidefiniteness of solutions of differential Lyapunov equations. Numerical examples that illustrate the benefits of this new method are given. In particular, numerical results for the efficient simulation of the weather phenomenon El Ni\~no are presented.
\end{abstract}


\end{frontmatter}


\section{Introduction}
Matrix differential equations arise quite naturally in many applications in science and engineering. Perhaps the most studied matrix differential equations are differential Riccati (DREs) and differential Lyapunov equations (DLEs) as they play a crucial role in many applications. For instance, they arise in optimal control problems like linear quadratic regulator and linear quadratic Gaussian design, $H_\infty$ control of linear time varying systems, optimal filtering, differential games, model reduction of linear time varying systems, damping optimization in mechanical systems, and control of shear flows subject to stochastic excitations. For an overview and more details, we refer to \cite{AboFIJ03,Ant05,PetUS00}.

In the literature, there is a large variety of approaches to compute the solution of DREs and DLEs, see, e.g., \cite{ChoL90b,Die92}. However, due to efficiency reasons, standard methods are not suitable for large-scale problems, in general. A major source of large-scale problems are partial differential equations (PDEs). In two and three space dimensions, their spatial discretization leads to systems with a significant number of degrees of freedom. The system matrices have often a particular structure, and they are usually sparse. In general, the arising differential equation are stiff which in turn requires integrators that handle the stiffness in an efficient way. 

The problem of solving large-scale DREs/DLEs has recently received considerable attention. Various integrators based on low-rank approximations have been developed. In particular, in~\cite{BennerMena1,BennerMena2} the authors proposed efficient numerical methods which are based on a matrix valued version of backward differentiation formulas (BDFs) and Rosenbrock methods, respectively. Further, a low-rank splitting method for large-scale DREs has been introduced in~\cite{Stillfjord}.
A different approach to compute low-rank approximations to large-scale matrix differential equations was proposed in \cite{KL07}. It is the so-called dynamical low-rank approximation. Its strength relies on the fact that a low-rank approximation to the solution is computed by solving differential equations only for its low-rank factors, see also \cite{KLW16, LO14, NL08}.
 
The performance of low-rank based methods relies on the decay of the singular values of the solutions. This phenomenon has been deeply studied and is frequent in applications, see, e.g., \cite{AntSZ02,Sab07} and references therein. 

In this paper, we focus on a new approach for solving DREs and DLEs that arise in PDE based models. 
Following the ideas of \cite{CKOR16} and \cite{LO14} we develop an efficient algorithm based on a splitting integrator. We split the vector field into a linear stiff and a non-linear non-stiff part. The stiff problem is efficiently integrated by an exponential integrator. We employ here the Leja method which proves to be very competitive, see~\cite{CKOR16}. We combine this integrator with the dynamical low-rank approximation for the non-stiff part. In particular, we employ the projector-splitting scheme, proposed in \cite{LO14}. This choice makes our low-rank integrator very robust with respect to small singular values. In this work we keep the rank fixed during time integration. Changing dynamically the rank would require a strategy for adapting the rank, based on a control of the corresponding error. Such a rank-adaptive approach is feasible but beyond the scope of our work. However, we illustrate the behaviour of our integrator with respect to different approximation ranks by numerical experiments.

The paper is organized as follows. In Section 2 we introduce the class of matrix differential equations considered and we propose our low-rank numerical integrator. In Sections 3 and 4, we specialize our integrator for DLEs  and DREs, and we discuss a possible extension to generalized DREs. We give details on the implementation and present numerical results that support our chosen approach. Finally, some conclusions are given.


\section{Matrix differential equations and numerical integrators} \label{integrator}
In this work we are interested in the numerical solution of the following class of matrix differential equations
\begin{align} \label{eq}
\dot{X}(t) = AX(t)+X(t)A\trasp+G(t,X(t)), \quad X(t_0) = X_0, \quad t \in [t_0,T],
\end{align}
where $X(t)$, $A \in \mathbb{R}^{d \times d}$, $G: \mathbb{R} \times \mathbb{R}^{d \times d} \rightarrow \mathbb{R}^{d \times d}$, and $(\cdot)\trasp$ denotes the transpose. This class includes, among others, differential Lyapunov and Riccati equations.
In particular, we consider here problems stemming from PDEs. Therefore, the matrix $A$ is typically the spatial discretization of a differential operator. The function $G$ on the other hand is assumed to be non-stiff.
By means of the variation-of-constants formula, the solution of \eqref{eq} can be written  as
\begin{equation}\label{eq:voc}
X(t) = \ee^{(t-t_0) A} X(t_0) \ee^{(t-t_0) A\trasp} + \int_{t_0}^t \ee^{(t-s)A} G(s,X(s)) \ee^{(t-s)A\trasp} \dd s
\end{equation}
for $t_0\le t\le T$.

The stiffness of \eqref{eq} is a limitation for most of the standard numerical integrators. Explicit methods have to use tiny step sizes whereas the use of implicit methods can result in high computational cost. Both approaches are thus not efficient. As a remedy, we design here a numerical integrator which is able to handle the stiff part of the solution \eqref{eq:voc} in an exact and efficient way. The key idea is the use of splitting methods. For an introduction to this class of numerical integrators we refer to \cite{HLW00}.
The structure of \eqref{eq} suggests a splitting into two terms, where we split the stiff linear part from the non-stiff non-linearity. Following this approach, two subproblems arise:
\begin{subequations}\label{sub}
	\begin{align}
	\label{sub1} \dot{M}(t) & = AM(t)+M(t)A\trasp, \\
	\label{sub2} \dot{N}(t) & = G(t,N(t)),
\end{align}
\end{subequations}
where $t \in [t_0,T]$.
A clear advantage of this splitting approach is that the stiffness is only present in the first subproblem \eqref{sub1}, whereas the second equation \eqref{sub2} is a non-stiff ordinary differential equation. Therefore, we have more freedom for the numerical integration of \eqref{sub2}.
Depending on the way of recombining the partial flows of \eqref{sub}, we obtain splitting methods with different orders of convergence.

Let us denote with $\Phi_{\tau}^A(Z)$ the exact solution of \eqref{sub1} at $t_0+\tau$  with initial value $M(t_0) = Z$, i.e.,
\begin{align} \label{exact_linear}
\Phi_{\tau}^A(Z) = M(t_0+\tau) = \ee^{\tau A} Z \ee^{\tau A\trasp}.
\end{align}
Further, let $\Phi^G_{\tau}(Z)$ denote the exact solution of \eqref{sub2} with initial value $N(t_0)=Z$, i.e., $\Phi_{\tau}^G(Z) = N(t_0+\tau)$.
The simplest splitting integrator for solving \eqref{eq} is then given by
\begin{align} \label{Lie}
\mathcal{L}_{\tau} Z = \Phi_{\tau}^G \circ \Phi^A_{\tau}(Z),
\end{align}
where $Z=X_0$. The result $\mathcal{L}_{\tau} X_0$ is the numerical approximation to $X(t_0+\tau)$. This is a first-order method and called Lie splitting in the literature.
A second-order method, the so-called Strang splitting, is given by the symmetric formula
\begin{align} \label{Strang}
\mathcal{S}_{\tau} Z = \Phi^A_{\tau/2}\circ \Phi_{\tau}^G \circ \Phi^A_{\tau/2}(Z).
\end{align}

The numerical evaluation of \eqref{exact_linear} can be computationally expensive if the dimension $d$ is large. In this case we propose to approximate the action of the flow $\Phi^A_{\tau}$ by means of a low-rank approximation.
Note that the differential equation \eqref{sub1} is rank preserving, i.e, the rank of the solution $M(t)$ does not depend on time (see \cite[Lemma 1.22]{HM96}).
A rank-$r$ decomposition ($r\leq d$) of the initial data $Z$ is typically given by a truncated singular value decomposition (SVD) of the form $USV\trasp$, where $U,V \in \mathbb{R}^{d \times r}$ have orthonormal columns and $S \in \mathbb{R}^{r \times r}$ is diagonal.
Using this approximation for the initial data $Z$, we obtain a rank-$r$ approximation $M_1$ to the exact solution $M(t_0+\tau) = \Phi_{\tau}^A(Z)$ by
\begin{align} \label{approx_linear}
M_1 = \ee^{\tau A} USV\trasp \ee^{\tau A\trasp}.\end{align}
Note that the action of the matrix exponential $\ee^{\tau A}$ to the skinny matrices $U$ and $V$ can be efficiently computed, for instance, by Taylor interpolation \cite{ALH11}, interpolation at Leja points \cite{CKOR16}, and Krylov subspace methods \cite{G13,S92}. We will compute these actions here with the Leja method.
The resulting matrices $\ee^{\tau A} U$ and $\ee^{\tau A} V$ can be orthogonalized by means of a QR decomposition such that $\ee^{ \tau A} U= \widetilde{U} R$, where $\widetilde{U} \in \mathbb{R}^{d \times r}$ has orthonormal columns and $R~\in~\mathbb{R}^{r \times r}$.  Similarly we have $\ee^{ \tau A} V = \widetilde{V} P$.
After defining $\widetilde{S}=RSP\trasp$, a rank-$r$ solution of \eqref{sub1} is given as $M_1 = \widetilde{U} \widetilde{S} \widetilde{V}\trasp$.

For the solution of \eqref{sub2} we make use of the dynamical low-rank approach proposed in \cite{KL07} (see also \cite{KLW16,LO14,NL08}). This is a differential equation based approach to efficiently compute low-rank approximations to time dependent large matrices or to solutions of large-scale matrix differential equations.
Let us denote by $\mnf = \mnf^{d \times d}$ the manifold of rank-$r$ matrices of dimension $d\times d$.
At any time $t$, a rank-$r$ approximation to the solution $N(t)\in \mathbb{R}^{d \times d}$ of \eqref{sub2} is a matrix $Y(t) \in \mnf$ defined by the following condition: for every $t$, the matrix $\dot{Y}(t) \in \mathcal{T}_{Y(t)}\mnf$ is such that
\begin{equation} 
\label{min}\lVert \dot{Y}(t)-G(t,Y(t)) \rVert_F = \min, 
\end{equation}
where $\mathcal{T}_{Y(t)}\mnf$ is the tangent space of $\mnf$ at the current state $Y(t)$ and $\lVert \cdot \rVert_F $ denotes the Frobenius norm. The minimization condition leads to a differential equation for $Y$ on the low-rank manifold, which has to be solved numerically. Standard integrators fail due to the presence of small singular values. Therefore, a particular integrator, the so-called projector-splitting integrator was developed in \cite{LO14}. Its convergence for non-stiff $G$ and its robustness with respect to the small singular values was proven in \cite{KLW16}. 

Condition \eqref{min} states that $\dot{Y}(t)$ is obtained by an orthogonal projection of $G(t,Y(t))$ onto the tangent space $\mathcal{T}_{Y(t)}\mnf$, i.e.
\[ \dot{Y}(t) = P(Y(t))G(t,Y(t)), \label{proj}\]
where $P(Y)$ is this orthogonal projection. The matrix $Y(t)$ is represented as $U(t)S(t)V(t)\trasp$, where $U,V \in \mathbb{R}^{d \times r}$ with orthonormal columns and $S \in \mathbb{R}^{r \times r}$. Note that $S$ is not assumed to be diagonal. Using the explicit form of the projector, which was computed in \cite{KL07}, the evolution equation for the approximation $Y$ is seen to be
\begin{equation} \label{eq_proj}
\dot{Y} = G(t,Y)VV\trasp -UU\trasp G(t,Y) VV\trasp +UU\trasp G(t,Y).\end{equation}
The projector-splitting integrator itself makes use of splitting methods. The right-hand side of \eqref{eq_proj} is split up into three subproblems. The integrator only deals with differential equations for the low-rank factors $U,S$ and $V$.
The practical algorithm for a general nonlinearity $G(t,Y)$ is given in \cite[Sect.~2.2]{KLW16}.
Starting from $U_0S_0V_0\trasp$, a rank-$r$ approximation to the initial data $N(t_0)$, one step of the integrator is as follows.
\begin{enumerate}
	\item[a.] Solve $\dot{K}(t) = G(t,K(t)V_0\trasp)V_0$ with initial value $K(t_0)= U_0 S_0$.
	Then orthonormalize $K(t_1)$ by a QR decomposition to get $K(t_1) = U_1\widehat{S}_1$, where $U_1 \in \mathbb{R}^{d \times r}$ has orthonormal columns and $\widehat{S}_1\in \mathbb{R}^{r \times r}$.
	\item[b.] Solve $ \dot{S}(t) = -U_1 \trasp G(t, U_1 S(t)V_0 \trasp) V_0$ with initial value $S(t_0)= \widehat{S}_1$.
	Set $\widetilde{S}_0 = S(t_1)$.
	\item[c.] Solve $\dot{L}(t) = G(t,U_1L(t)\trasp)\trasp U_1$ with initial value $L(t_0)=V_0\widetilde{S}_0\trasp$.
	Then orthonormalize $L(t_1)$ by a QR decomposition to get $L(t_1) = V_1S_1\trasp$, where $V_1 \in \mathbb{R}^{d \times r}$ has orthonormal columns and $S_1\in \mathbb{R}^{r \times r}$.
\end{enumerate}
The above scheme can be tailored to the particular form of $G$. More details are given in the sections below. 
A detailed convergence analysis of the proposed integrator will be presented elsewhere.


\section{Differential Lyapunov equations}\label{sect:DLE}
Differential Lyapunov equations arise for $G(t,X) = Q$ with $Q$ being a constant matrix. Thus, we consider the following initial value problem
\begin{equation} \label{DLE}
\begin{aligned}
\dot{X}(t) & = A X(t) + X(t)A\trasp + Q, \\
X(t_0) & = X_0 \, ,
\end{aligned}
\end{equation}
where $X(t)$, $A$, $Q$ $\in \mathbb{R}^{d \times d}$, and $t \in [t_0,T]$. The matrix $A$ typically arises from the discretization of a differential operator. Further, $Q$ and the initial data $X_0$ are symmetric and positive semidefinite.
Since \eqref{DLE} is a linear differential equation with constant coefficients the solution exists for all times. Moreover, the solution is also symmetric and positive semidefinite. This is a straightforward consequence of~\eqref{eq:voc}, see also \cite{AboFIJ03}.

\subsection{A low-rank split-step integrator}  \label{integrator_lyap}
As explained in Section \ref{integrator} we split \eqref{DLE} into the following two subproblems:
\begin{subequations}
	\begin{align}
	\label{sub1_lyap} \dot{M}(t) = AM(t)+M(t)A\trasp, \quad & M(t_0) = M_0, \\
	\label{sub2_lyap} \dot{N}(t) = Q, \quad & N(t_0)= N_0, \end{align}
\end{subequations}
where $t \in [t_0,T]$.
Note that the exact solution of \eqref{sub2_lyap} simply is $N(t)=N_0+~tQ$. Employing this representation directly in our splitting method, however, would require one additional SVD per time step. To avoid this computational overhead, we follow the approach chosen in \cite{KLW16} and use the projector-splitting method.

Since the solution of \eqref{DLE} is symmetric and positive semidefinite we represent it as a rank-$r$ matrix of the form $USU\trasp$, where $U \in \mathbb{R}^{d \times r}$ has orthonormal columns and $S \in \mathbb{R}^{r \times r}$ is symmetric. This is an SVD-like decomposition with the additional constraint of symmetry.

The linear problem \eqref{sub1_lyap} can be treated as explained in Section \ref{integrator}.
The projector-splitting integrator for the solution of \eqref{sub2_lyap} in its standard formulation, however, does not preserve the symmetry and positive semidefiniteness of the solution.
Let $U_0 S_0 U_0 \trasp$ be a low-rank decomposition of the initial data $N_0$. In order to preserve the symmetry of the problem we propose a modification of the algorithm given at the end of Section \ref{integrator}, see also \cite[Sect.~2.2]{KLW16}. The modified projector-splitting integrator is described in step~\ref{step5} of Algorithm~\ref{algorithm_Lie}.
In the first substep we update the value of the left-sided orthonormal factor $U_0$ to $U_1$, whereas we keep the right-sided one till the last step. There we get $L(t_1) = S_1\widetilde {U}_1\trasp$. Imposing $\widetilde{U}_1 = U_1$ and taking advantage of the orthonormality of $U_1$ we obtain $S_1 = L(t_1)U_1$. The low-rank approximation $N_1$ at $t_0+\tau$ is then simply $U_1S_1U_1\trasp$.

In Algorithm \ref{algorithm_Lie} we summarize the above proposed Lie splitting for the solution of \eqref{DLE}. Note that the flows of all these substeps can be computed exactly. Moreover, the whole time integration is performed by working only with the low-rank factors of the solution, which has several computational advantages such as less computing time and less memory requirements.
By using the explicit structure of the subproblems, it is easy to show that $S_1$ and consequently $N_1$ are symmetric and positive semidefinite.

\begin{algorithm}
	\caption{The first-order low-rank split-step integrator for DLEs}
	\label{algorithm_Lie}
	\begin{algorithmic}[1]
		\State Compute a symmetric rank-$r$ approximation $M_0$ to the given initial data $X_0$: $M_0 = U_0 S_0 U_0\trasp$.
		\State Let $t_k = t_0+k\tau$ for $k \in \mathbb{N}$ and set $n = 0$.
		\State Compute $\ee^{\tau A} U_n$ and orthogonalize it by a QR decomposition to get $\ee^{\tau A} U_n = U_n^A R$, where $U_n^A \in \mathbb{R}^{d \times r}$ and $R \in \mathbb{R}^{r \times r}$.
		\State Define $S_n^A = RS_nR\trasp$.
		\State Given $U_n^A$ and $S_n^A$ perform one integration step:\label{step5}
		\begin{enumerate}
			\item[a.]  Solve $\dot{K}(t) = QU_n^A$ with initial value $K(t_n)= U_n^A S_n^A$.
			Then orthogonalize $K(t_{n+1})$ by a QR decomposition and set $U_{n+1}\widehat{S}_{n+1} = K(t_{n+1})$, where $U_{n+1} \in \mathbb{R}^{d \times r}$ has orthonormal columns and $\widehat{S}_{n+1}\in \mathbb{R}^{r \times r}$.
			\item[b.] Solve $\dot{S}(t) = -U_{n+1}\trasp QU_n^A$ with initial value $S(t_n)= \widehat{S}_{n+1}$.
			Then set $\widetilde{S}_n = S(t_{n+1})$.
			\item[c.] Solve $ \dot{L}(t) = U_{n+1}\trasp Q$ with initial value $L(t_n)=\widetilde{S}_n(U_n^A)\trasp$. Then set $S_{n+1} = L(t_{n+1})U_{n+1}$.
		\end{enumerate}
		\State Increase $n$ by 1.
		\State Go to 3 and repeat as long as $t_n < T$.
	\end{algorithmic}
\end{algorithm}

\begin{lemma} \label{lemma}
	The matrix $S_n$ arising in step~\ref{step5}c of Algorithm \ref{algorithm_Lie} is symmetric and positive semidefinite.
\end{lemma}
\begin{proof}
	Without loss of generality we prove the lemma for $n=0$. For simplicity, we will also omit the upper index $A$ and write $U_0$ instead of $U_0^A$, etc.
	Starting from the rank-$r$ approximation $U_0S_0U_0\trasp$ of $N_0$, the first step of the projector-splitting integrator gives
	\begin{equation} \label{first_expl}
	U_1\widehat{S}_1 = U_0S_0+\tau QU_0.
\end{equation}
	The solution of the second substep is
	\begin{equation*} 
	\widetilde{S}_0 = \widehat{S}_1-\tau U_1 \trasp Q U_0. 	
	\end{equation*}
	From step~\ref{step5}c we get
	\begin{equation}\label{second_expl}
    S_1 = \widetilde{S}_0 U_0\trasp  U_1 + \tau U_1 \trasp Q  U_1.
    \end{equation}
	Using the orthonormality of $U_1$, we compute $\widehat{S}_1 = U_1\trasp U_0S_0+\tau U_1 \trasp QU_0$ from \eqref{first_expl}. Inserting this formula into \eqref{second_expl} we get the symmetric formula
	\[ S_1 = U_1\trasp U_0 S_0 U_0\trasp U_1 + \tau U_1 \trasp Q  U_1\]
	which proves the assertion.
\end{proof}

\subsection{Numerical results for a parabolic problem} \label{example_dle}
In order to illustrate the behavior of the integrator, we take a simple test problem.
We consider the heat equation
\[\partial_t w = \Delta w, \qquad w|_{\partial \Omega}=0\]
on $\Omega = (0,1)^2$ with homogeneous Dirichlet boundary conditions.
The associated DLE is of the form \eqref{DLE}, where
$A$ arises from the spatial discretization of the Laplacian. Thus, we have $A = \widetilde{A} \otimes I + I \otimes \widetilde{A}$, where $\widetilde{A}$ is the 1D standard centered finite differences matrix with $\widetilde{d}$ uniformly spaced grid points in the interior. In the following we take $\widetilde{d} = 20$ and obtain $d = \widetilde{d}^2 = 400$.
The matrix $Q$ in our example has random coefficients and is of rank 5. The initial value $X_0$ is chosen to be a random matrix of rank 10. 
This choice ensures the low-rank behaviour of the solution of the DLE. This can be observed in Figure \ref{fig:lyap_rank_lie}, left where we plot the first 90 singular values of a numerical solution computed with DOPRI5 \cite{HNW93} at the final integration time $T=0.1$. This solution is taken as the reference solution for all tests in this section. The code DOPRI5 is an explicit Runge--Kutta method of order~5 with adaptive step size strategy.

\begin{figure}[t]
	\begin{minipage}[b]{.5\textwidth}
		\centering
		\includegraphics[width = \columnwidth]{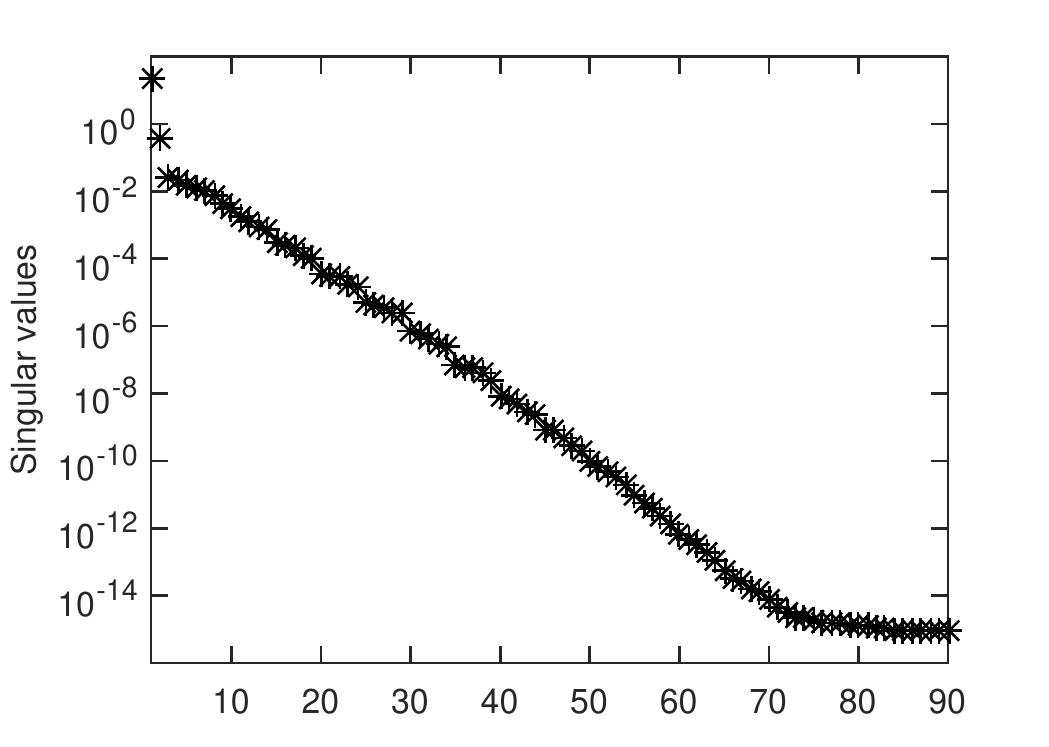}
	\end{minipage}
	\hspace{2mm}
	\begin{minipage}[t]{.5\textwidth}
		\centering
		\includegraphics[width = \columnwidth]{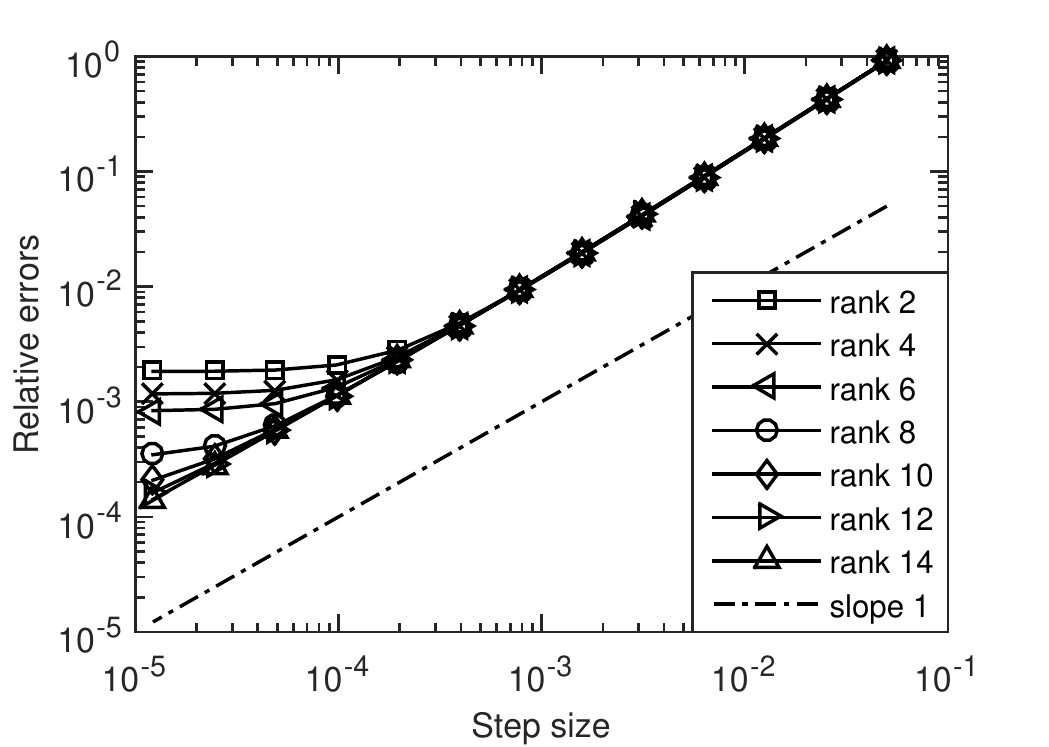}
	\end{minipage}
	\caption{Results for the DLE of Section \ref{example_dle} for $d= 400$. Left: First 90 singular values of the reference solution computed with DOPRI5 at $T=0.1$. Right: Errors of Lie splitting described in Algorithm \ref{algorithm_Lie} in the Frobenius norm \eqref{norm} as a function of step size and rank at $T=0.1$.}
	\label{fig:lyap_rank_lie}
\end{figure}

\begin{figure}[t]
	\begin{minipage}[t]{.5\textwidth}
		\centering
		\includegraphics[width = \columnwidth]{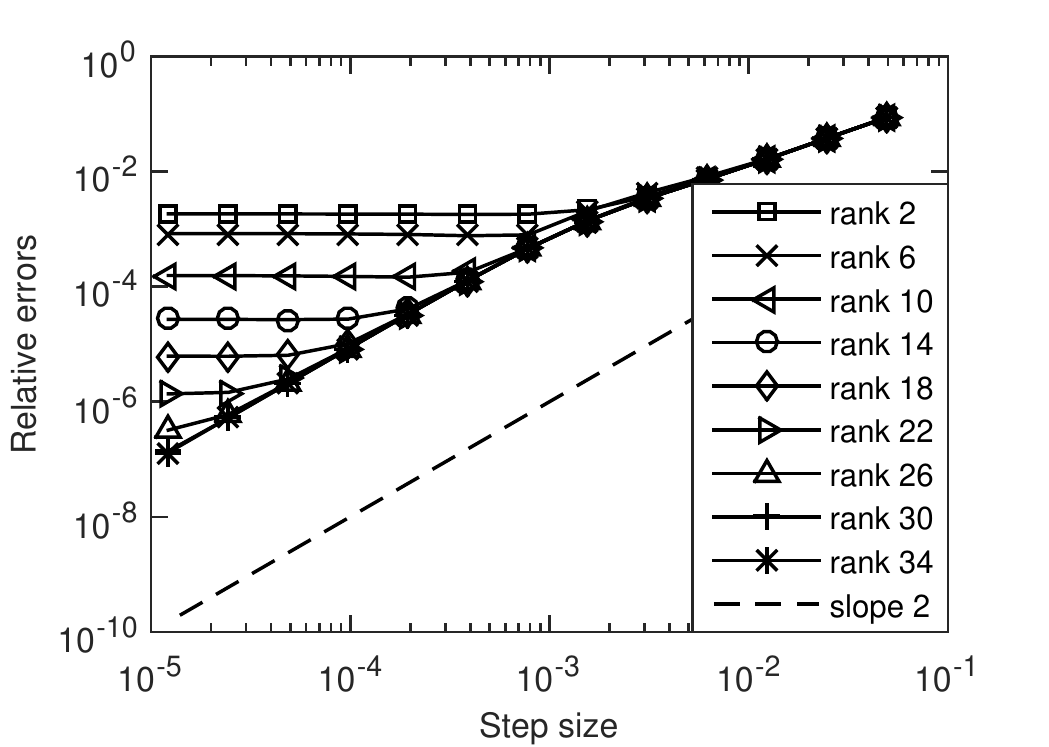}
	\end{minipage}
	\hspace{2mm}
	\begin{minipage}[t]{.5\textwidth}
		\centering
		\includegraphics[width = \columnwidth]{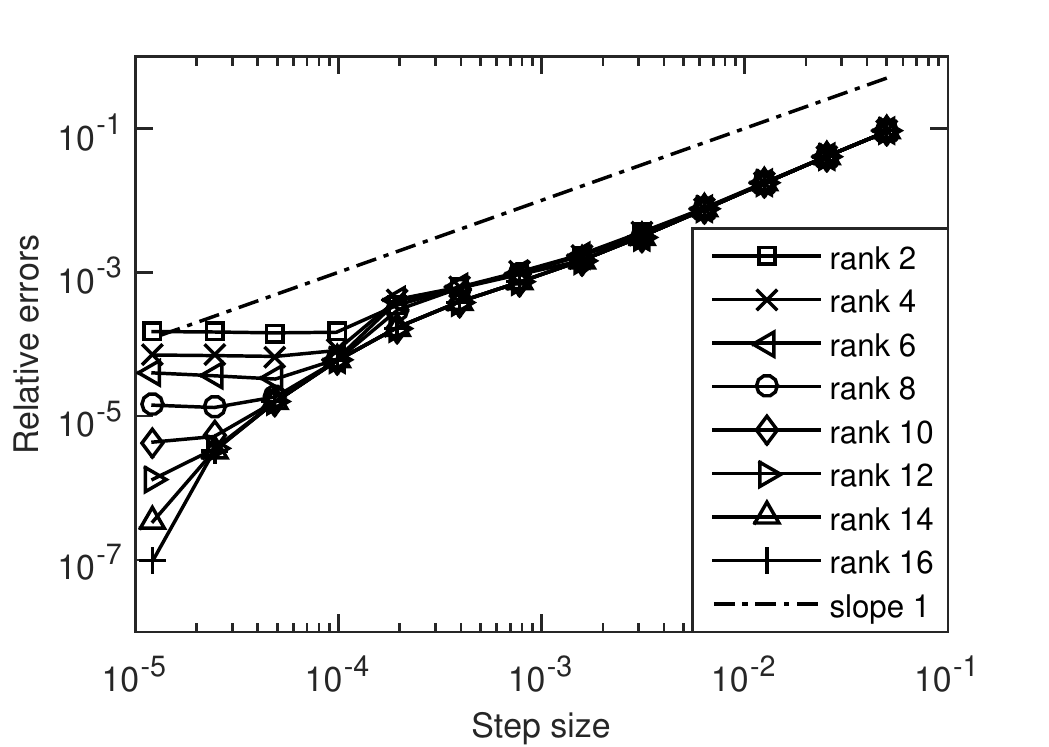}
	\end{minipage}
	\caption{Errors of Strang splitting for the solution of the DLE of Section \ref{example_dle}. The errors are measured in the Frobenius norm \eqref{norm} as a function of step size and rank at $T=0.1$. Left: $d = 400$, right: $d = 3600$.}
	\label{fig:lyap_strang}
\end{figure}

In Figure \ref{fig:lyap_rank_lie}, right we show the error behaviour of the Lie splitting \eqref{Lie} as described in Algorithm \ref{algorithm_Lie}. The errors are measured in an appropriately scaled Frobenius norm
\begin{equation} \label{norm}
\lVert Z \rVert_F =  \frac{1}{d} \sqrt{\sum_{i,j = 1}^{d}Z_{ij}^2}, \qquad Z \in \mathbb{R}^{d \times d}.
\end{equation}
We observe that the error of our method is composed by two different contributions, the error due to the outer splitting into \eqref{sub1_lyap} and \eqref{sub2_lyap} and the error due to the low-rank approximation. As long as the error due to the low-rank approximation is not dominant, we observe the expected order of convergence one for Lie splitting. On the other hand, if the low-rank approximation is poor, decreasing the step size will not improve the quality of the solution. We observe a stagnation of the error around certain values depending on the rank. These errors can be related with the magnitude of the first singular value discarded for each choice of the approximation rank, as can be observed by comparing Figure \ref{fig:lyap_rank_lie}, left and right.
In Figure \ref{fig:lyap_strang} we show the corresponding results for Strang splitting for $d = 400$ and $d = 3600$, respectively. On the left we observe the expected order of convergence two for the outer splitting. Again, when the approximation rank is too low, the outer error becomes independent of any step size refinement.
An order reduction shows up in the figure on the right. It is due to the fact that the inhomogeneity $Q$ is not in the domain of the Laplacian. Therefore, $AQ$ cannot be bounded independently of $d$, see \cite{EO15}.

In Table \ref{tab:sym} we list the defects in symmetry
\begin{equation} \label{def_sym}
d_{\text{sym}} = \frac{\lVert Y-Y\trasp \rVert_F}{\lVert Y_{\text{ref}}\rVert_F}, \end{equation}
where $Y$ is the solution obtained with Lie splitting as defined in Algorithm \ref{algorithm_Lie} and $Y_{\text{ref}}$ is the reference solution. We observe that the symmetry is numerically preserved, as expected from Lemma \ref{lemma}.
In Table \ref{tab:non_sym} we show that the standard non-symmetric integrator as presented in Section \ref{integrator} is not symmetry-preserving. It is therefore necessary to modify the projector-splitting as explained in detail in Algorithm \ref{algorithm_Lie}.

Due to Lemma~\ref{lemma}, our integrator also preserves positive semidefiniteness of the solution. To show this, we denote with $\widehat{Y}$ the nearest symmetric positive semidefinite matrix to $Y$ obtained by the method proposed in \cite{H88}. In Table \ref{tab:spd} we list the defects in positive semidefiniteness
\begin{equation} \label{def_psd}
d_{\text{psd}} = \frac{\lVert Y-\widehat{Y} \rVert_F}{\lVert Y_{\text{ref}}\rVert_F}
\end{equation}
for different approximation ranks and step sizes.

\begin{table}[h]
	\centering
	\begin{tabular}{c|c|c|c|c|c}
		& ns = $2$ &  ns = $2^4$ &  ns = $2^7$ &  ns = $2^{11}$ &  ns = $2^{13}$ \\\hline
		rk = 2 & 4.7294e-16&1.6989e-16&1.5067e-16&7.3221e-16&5.1379e-16\\\hline
		rk = 4 & 7.7743e-16&3.8092e-16&4.9242e-16&1.0861e-15&2.7226e-15\\\hline
		rk = 6 & 1.0498e-15&3.4863e-16&5.0365e-16&3.2494e-16&3.9935e-16\\\hline
		rk = 8 & 1.0781e-15&4.1657e-16&3.9016e-16&2.6837e-15&1.2762e-14\\\hline
		rk = 10 & 8.2787e-16&7.7903e-16&1.3127e-15&1.7855e-15&2.8277e-15\\\hline
		rk = 12 & 1.1677e-15&4.6731e-16&6.3586e-16&2.1097e-15&6.9804e-15\\\hline
		rk = 14 & 1.8577e-15&5.5461e-16&8.0892e-16&1.9972e-15&2.5059e-15\\
	\end{tabular}
	\caption{Defects in symmetry for Algorithm \ref{algorithm_Lie} applied to the DLE of Section \ref{example_dle} for $d= 400$. The defect \eqref{def_sym} is displayed as a function of the approximation rank (rk) and the number of time steps (ns) at $T=0.1$.}
	\label{tab:sym}
\end{table}

\begin{table}[ht]
	\centering
	\begin{tabular}{c|c|c|c|c|c}
		& ns = $2$ &  ns = $2^4$ &  ns = $2^7$ &  ns = $2^{11}$ &  ns = $2^{13}$ \\\hline
		rk = 2 & 9.5178e-03&3.2545e-03&4.2299e-04&4.3455e-05&5.2796e-06\\\hline
		rk = 4 & 1.2097e-02&3.8332e-03&2.4587e-04&2.0065e-05&2.3394e-06\\\hline
		rk = 6 & 2.3350e-04&4.1849e-03&2.6199e-04&2.0735e-05&2.3876e-06\\\hline
		rk = 8 & 1.2875e-05&8.0839e-04&1.9463e-04&1.5981e-05&1.8315e-06\\\hline
		rk = 10 & 2.5457e-07&1.5020e-04&1.0291e-04&8.1713e-06&9.2641e-07\\\hline
		rk = 12 & 1.3272e-09&5.9891e-05&4.7831e-05&3.9946e-06&4.5438e-07\\\hline
		rk = 14 & 1.0354e-12&5.5624e-06&1.8540e-05&1.7506e-06&1.9895e-07\\
	\end{tabular}
	\caption{Defects in symmetry for Lie splitting of Section \ref{integrator} applied to the DLE of Section \ref{example_dle} for $d= 400$ without symmetry-preserving modification. The defect \eqref{def_sym} is displayed as a function of the approximation rank (rk) and the number of time steps (ns) at $T=0.1$.}
	\label{tab:non_sym}
\end{table}

\begin{table}[ht]
	\centering
	\begin{tabular}{c|c|c|c|c|c}
		& ns = $2$ &  ns = $2^4$ &  ns = $2^7$ &  ns = $2^{11}$ &  ns = $2^{13}$ \\\hline
		rk = 2 & 4.9230e-15&2.7584e-15&2.7871e-15&4.3101e-15&2.3554e-15\\\hline
		rk = 4 & 4.1822e-15&4.6147e-15&3.2650e-15&5.4743e-15&3.2069e-15\\\hline
		rk = 6 & 5.9181e-15&1.9928e-15&2.5106e-15&3.0357e-15&4.0661e-15\\\hline
		rk = 8 & 5.1317e-15&2.6077e-15&1.7991e-15&3.0388e-15&6.6778e-15\\\hline
		rk = 10 & 7.9000e-15&3.3707e-15&2.7210e-15&3.3223e-15&2.3704e-15\\\hline
		rk = 12 & 5.7540e-15&2.2482e-15&2.0631e-15&2.6968e-15&5.7000e-15\\\hline
		rk = 14 & 4.2693e-15&7.0395e-15&2.6659e-15&2.0693e-15&3.2132e-15
	\end{tabular}
	\caption{Defects in positive semidefiniteness for Algorithm \ref{algorithm_Lie} applied to the DLE of Section \ref{example_dle} for $d= 400$. The defect \eqref{def_psd} is displayed as a function of the approximation rank (rk) and the number of time steps (ns) at $T=0.1$.}
	\label{tab:spd}
\end{table}

\subsection{Comparison with a standard approach} \label{comparison}
In order to give more insight into the performance of the low-rank splitting described in Algorithm \ref{algorithm_Lie} we carry out a comparison with a standard method. 
The basic idea is to discretize the DLE in time  to convert the differential problem into an algebraic one.
Since our aim is to compare the first-order low-rank splitting, we apply here the backward Euler method with time step $\tau$ to obtain the numerical approximation $X_{n+1}$ to $X(t_{n+1})$. This gives   
\[ X_{n+1} = X_n + \tau \left(A X_{n+1} + X_{n+1}A\trasp + Q\right).  \] 
After recombining the terms, we end up with the following algebraic Lyapunov equation (ALE)
\begin{equation} \label{ALE}
0 = (X_n + \tau Q) + \left(\tau A-\frac{1}{2}I\right)X_{n+1}+X_{n+1}\left( \tau A\trasp-\frac{1}{2}I\right).
\end{equation} 
Several approaches are available in the literature for solving this type of equations. We employ the one proposed in \cite{S07}. 
It consists of projecting the ALE onto an approximation space, generated as combination of Krylov subspaces in $A$ and $A^{-1}$. Such an approach is also called extended Krylov method. The reduced equation is then solved by means of a direct solver. 

In particular, consider an ALE of the form
\begin{equation} \label{factor_ale} 0 = \widetilde{A}X+X\widetilde{A}\trasp +\widetilde{B} \widetilde{B} \trasp, \end{equation}
with $\widetilde{A} \in \mathbb{R}^{d\times d} $ and $\widetilde{B} \in \mathbb{R}^{d \times s}$, $s \ll d$.  
Let us denote with $\widetilde{V}$ the matrix whose columns span the approximation space and with $Y$ the solution of the projected ALE.
Then the solution of the original ALE is reconstructed as \[X = \widetilde{V} Y \widetilde{V}\trasp.\]
A low-rank solution is generated discarding all the eigenvalues of $Y$ which are smaller than a certain tolerance \texttt{tolY}. Then the solution $X$ is given in low-rank form as $ZZ\trasp$. 
A second user-supplied parameter, denoted by \texttt{tol}, is used for the stopping criterion of the algorithm. The iteration is stopped if
\[ 
\frac{\lVert \widetilde{A}X^m+X^m\widetilde{A}\trasp +\widetilde{B} \widetilde{B} \trasp \rVert_2}{2 \lVert \widetilde{A}\rVert_F \lVert Y^m\rVert_F +\lVert \widetilde{B}\rVert_F^2} \leq \texttt{tol},
\] 
where the superscript $m$ denotes the number of the iteration. The algorithm associated with these procedure is called K-PIK (Krylov-plus-inverted Krylov). For further details we refer to \cite{S07}. 

In the following we report some numerical experiments for the example we considered in the beginning of Section \ref{example_dle}.
We solve it until $T=0.1$. An algebraic equation of the form \eqref{ALE} has to be solved for each time step. 
In order to use the K-PIK procedure we set 
\[ \widetilde{A} = \tau A -\frac{1}{2}I \quad \text{and} \quad \widetilde{B} = [\sqrt{\tau} R \; \; Z_n], \]
where $Z_n$ is the low-rank factor of the solution $X_n$ at time $t_n$, i.e., $X_n = Z_nZ_n\trasp$ and $R$ is the low-rank factor of $Q$, i.e., $Q = R R\trasp$. 

In Figure \ref{fig:sim_lyap} we illustrate the behaviour of this integrator for different values of \texttt{tolY}. The results are all obtained in Matlab with $\texttt{tol} = 10^{-12}$. The projected ALE is solved with Matlab's $\texttt{\textsf{lyap}}$ function. Moreover, as a comparison, we compute the solution of \eqref{DLE} by directly solving \eqref{ALE} by the routine \texttt{lyap}. Since in both the cases the problem is discretized in time in the same way, this comparison reduces to a comparison between the K-PIK procedure and the \texttt{lyap} method. 
In Figure \ref{fig:sim_lyap}, left we show the error behaviour of the two methods. The K-PIK procedure is tested for different values of \texttt{tolY}. The error we observe is basically just the first order error of the 
backward Euler method. Only when the time step becomes very small ($\tau \approx 10^{-6}$) we observe an additional error if \texttt{tolY} is not chosen sufficiently small. This is basically the error due to the neglect of some eigenvalues of the solution. 
In Figure \ref{fig:sim_lyap}, right we observe how the choice of \texttt{tolY} influences the rank of the solution. The full rank given by the solution computed with \texttt{lyap} is recovered by the K-PIK procedure with \texttt{tolY}= $10^{-12}$.

In Figure \ref{fig:comparison} we compare the results obtained with the Lie splitting approach described in Algorithm \ref{algorithm_Lie} with the ones obtained from the combination of the backward Euler method with the K-PIK procedure. 
In Figure \ref{fig:comparison}, left the relative errors for the two methods are shown. We observe that the first-order splitting method is more accurate than the backward Euler method for this example. This is due to the fact that we solve the linear subproblem exactly. 

\begin{figure}[h]
	\begin{minipage}[t]{.5\textwidth}
		\centering
		\includegraphics[width = \columnwidth]{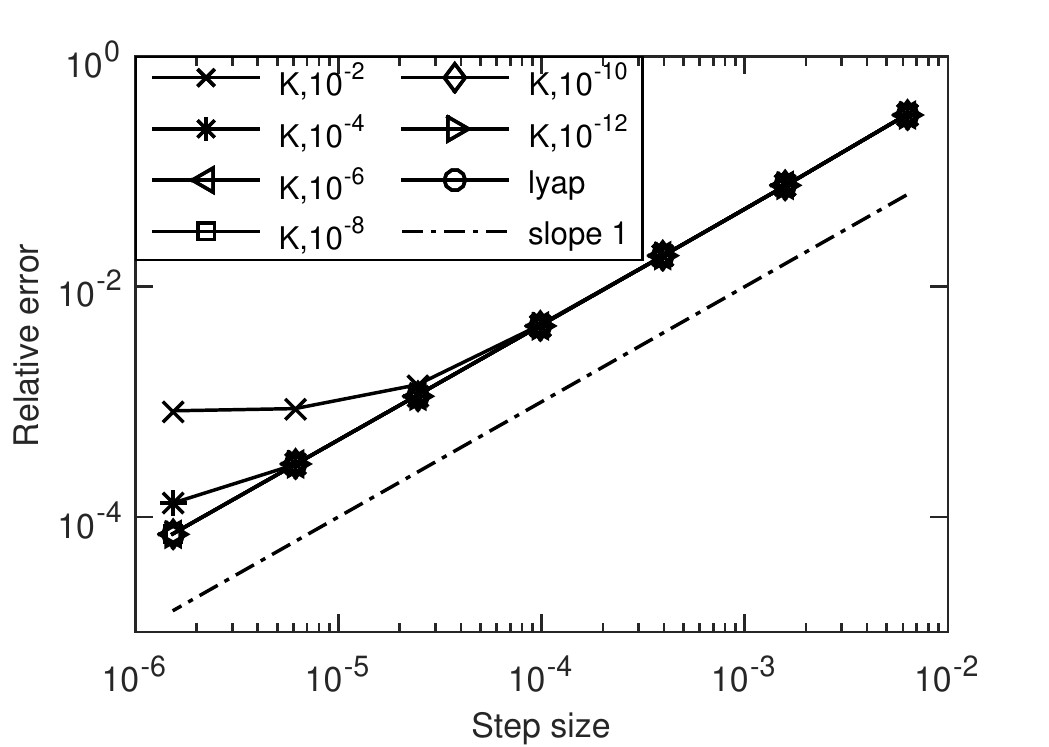}
	\end{minipage}
	\hspace{2mm}
	\begin{minipage}[t]{.5\textwidth}
		\centering
		\includegraphics[width = \columnwidth]{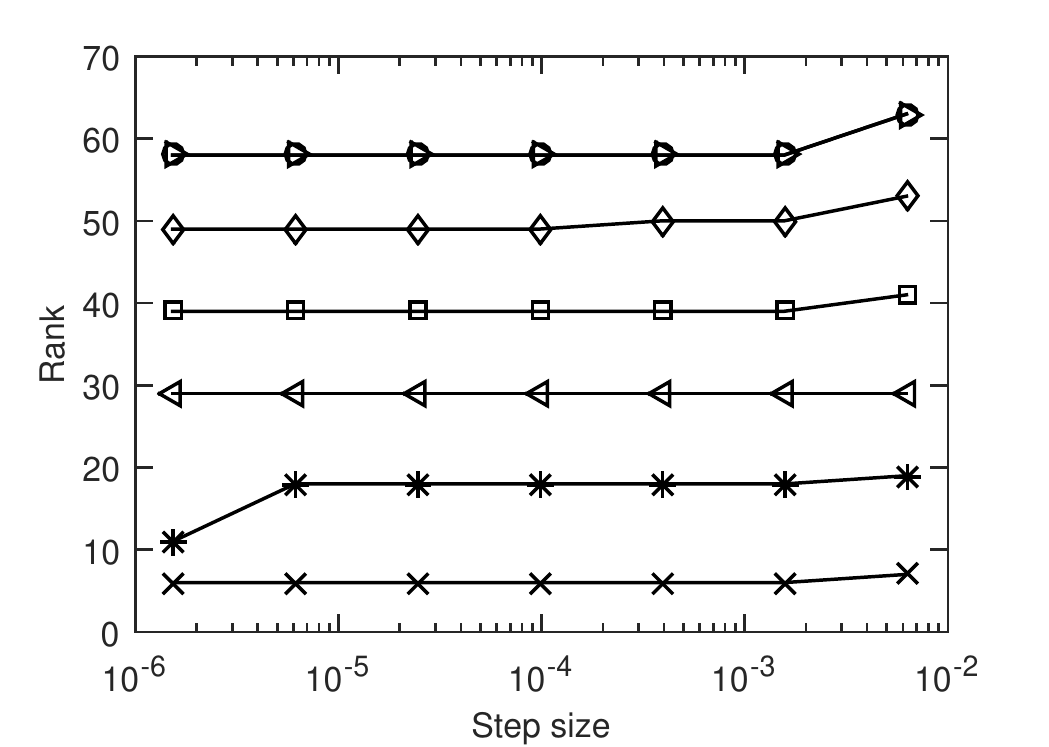}
	\end{minipage}
	\caption{Comparison between the K-PIK procedure for different values of \texttt{tolY} (K, \texttt{tolY}) and \texttt{lyap}. Left: Errors in Frobenius norm \eqref{norm} as a function of step size  at $T = 0.1$. Right: Rank as a function of the step size at $T=0.1$.}
	\label{fig:sim_lyap}
\end{figure}

\begin{figure}[h]
	\begin{minipage}[t]{.5\textwidth}
		\centering
		\includegraphics[width = \columnwidth]{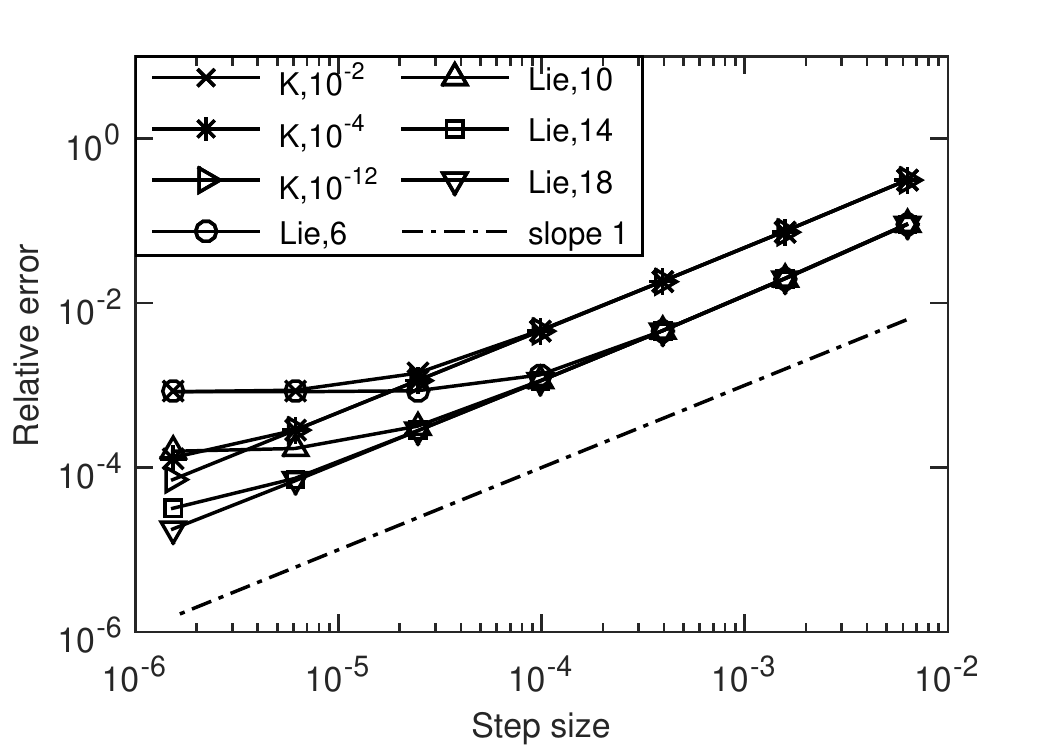}
	\end{minipage}
	\hspace{2mm}
	\begin{minipage}[t]{.5\textwidth}
		\centering
		\includegraphics[width = \columnwidth]{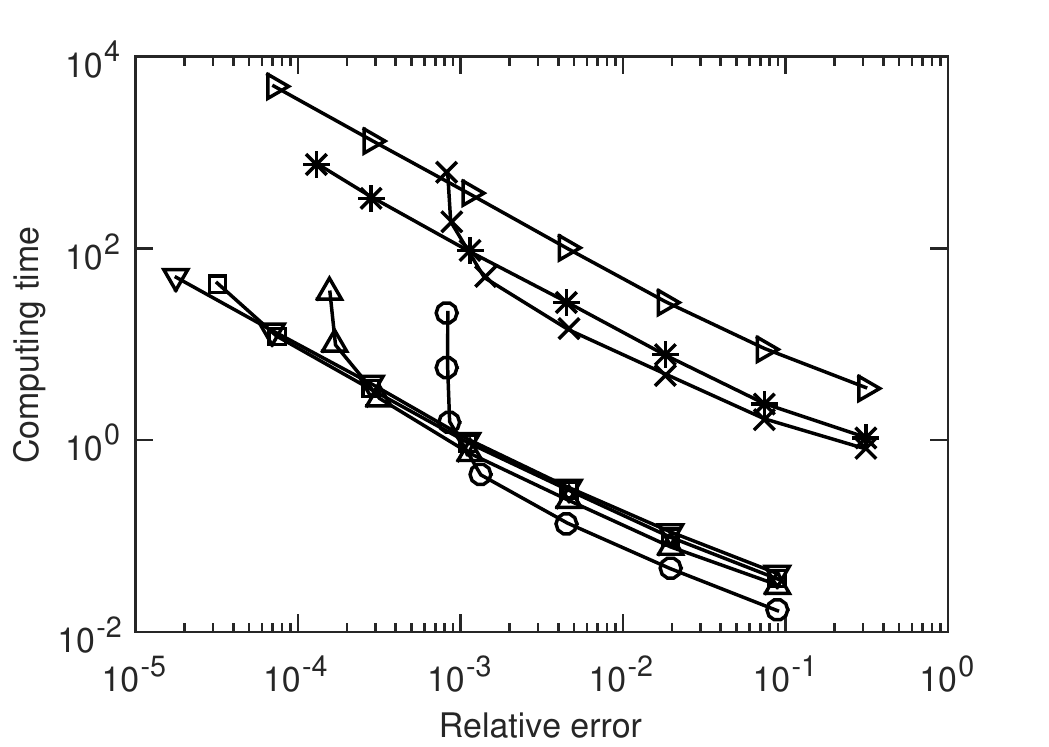}
	\end{minipage}
	\caption{Comparison between the Lie splitting described in Algorithm \ref{algorithm_Lie} for different ranks (Lie, rank) and the backward Euler method combined with the K-PIK procedure for different values of \texttt{tolY} (K, \texttt{tolY}). Left: Errors in the Frobenius norm \eqref{norm} as a function of step size at $T = 0.1$. Right: Computing time as a function of the error at $T = 0.1$.}
	\label{fig:comparison} 
\end{figure}

In Figure \ref{fig:comparison}, right the computing time at $T =0.1$ is displayed. 
The solution of the ALE \eqref{ALE} is computed by the K-PIK procedure, where the Choleski decomposition of the matrix $\widetilde{A}$ is computed once and for all. The computation of the matrix exponential required for the Lie splitting in Algorithm \ref{algorithm_Lie} was carried out by Leja interpolation \cite{CKOR16} with single precision. We conclude that the splitting proposed here is considerably faster than the backward Euler method combined with the K-PIK approach.

\subsection{Numerical results for the simulation of El Ni\~no}
The quasiperiodic weather phenomenon El Ni\~no which is characterized by an unusual warming of the sea surface in the Indo-Pacific ocean, has a huge impact on the climate worldwide and is also responsible for many natural disasters.  Different models (both deterministic and stochastic) are used to describe the variation in the sea surface temperature. We consider in the following a stochastic advection equation driven by additive noise
\begin{align}
\dd X(t) = \mathcal{A} X(t) + F(t) , \qquad t\in [t_0,T],
\end{align} 	
where the vector $X$ contains the sea surface temperature (SST) anomalies, the operator $\mathcal{A} = u \cdot \nabla$ describes the advection by the ocean currents $u$, and $F$ is a random vector which aggregates external forces like wind stress or evaporation \cite{C09,PS95}. Following \cite{PS95} we construct the SST anomalies, using the dataset OISST \cite{cisl}, whereas for the ocean currents we use the dataset OSCAR \cite{OSCAR}, both from the National Oceanic and Atmospheric Administration (NOAA). We model the random force $F$ by a Gaussian white noise process with zero mean and covariance operator $\mathcal{Q}$ and compute the first two moments by
\begin{align}
\dd\widehat{X}(t) &= \mathcal{A} \widehat{X}(t) , \label{exp} \\
\dd \mathcal{P}(t) &= \mathcal{A} \mathcal{P}(t) + \mathcal{P}(t) \mathcal{A}\trasp + \mathcal{Q}, \label{lde}
\end{align}
where $t \in[t_0,T]$ and we denote by $\widehat{X}$ the expectation and by $\mathcal{P}$ the covariance of $X$ \cite{MP17}. As the solution $X$ is a Gaussian random field, it is completely defined by its second-order statistics.

We discretize the operator $\mathcal{A}$ via centered finite differences and obtain the matrix $A$. The first subproblem \eqref{exp} is solved by a method based on Leja interpolation. Discretizing the Indo-Pacific ocean we get $3900$ grid points, hence solving the large-scale differential Lyapunov equation \eqref{lde} with full rank is rather expensive. For such computations, our splitting algorithm has substantial advantages. It requires considerably less computing time and memory. In order to illustrate its accuracy, we make a comparison of the proposed algorithm with a second order scheme. Similarly as in Section \ref{comparison} we apply the backward Euler method to the DLE \eqref{lde} and solve the resulting ALE by the K-PIK procedure. In order to obtain a second order scheme we then use Richardson extrapolation. Given $Z_n$ which is a low-rank factor of $P_n$, we make one step with step size $\tau$ and get an approximation $\tilde{Z}_{n+1}$ and two steps with step size $\frac{\tau}{2}$ to obtain $\hat{Z}_{n+1}$. The numerical approximation $Z_{n+1}$ is then given by
\begin{align}
Z_{n+1} =  2\hat{Z}_{n+1} - \tilde{Z}_{n+1}. \label{rich}
\end{align} 

As discussed in Section \ref{comparison} one has to choose parameters \texttt{tol} and \texttt{tolY}, which can have an impact in the solution. In the following we use $\texttt{tol} = 10^{-10}$ and $\texttt{tolY} = 10^{-12}$. Figure \ref{fig:richardson} shows the relative error of this approach. The reference solution is obtained by the same method but with step size $2^{-8}$, which is 8 times smaller than the smallest step size used in Figure \ref{fig:richardson}.
\begin{figure}[h] 
	\centering 
	\includegraphics[width = 0.5\columnwidth]{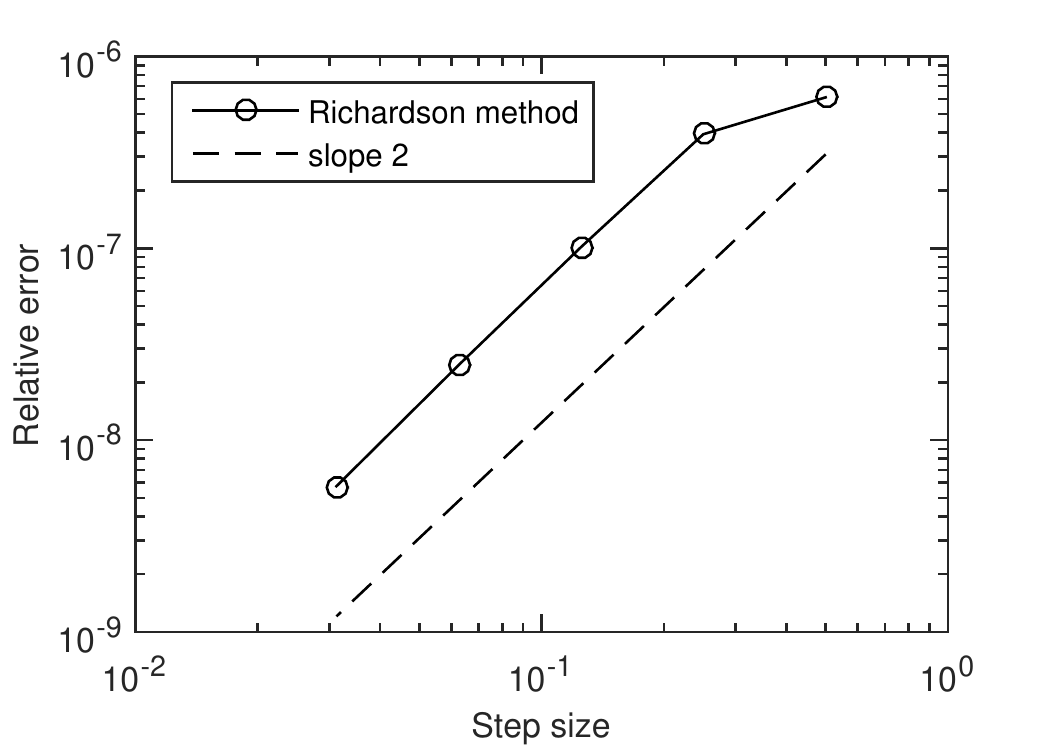} 
	\caption{Error at $T = 5$ of the method \eqref{rich} for different step sizes and $\texttt{tolY} = 10^{-12}$ and $\texttt{tol} = 10^{-10}$. }
	\label{fig:richardson}
\end{figure}

In Table~\ref{tab:nino} we compare the error in Frobenius norm \eqref{norm} of the low-rank integrator with respect to the second order method \eqref{rich} for the years $2013$ (no event) and $2015$ (a strong event), respectively. In both experiments, we took as starting time June~15. In the table we display the absolute errors of the low-rank approximation after one, two and three weeks for two different approximation ranks. We observe that the error gets smaller for rank 100. For higher ranks, however, the splitting error dominates in this example.

\begin{table}[htbp]
	\centering
	\begin{tabular}{c|c|c||c|c|c}
	\multicolumn{3}{c||}{2013} & \multicolumn{3}{|c}{2015} \\
	\hline 
		 weeks & rank &error  &	 weeks & rank & error  \\
		\hline
		 1 & 10 & 3.1117e-05 &  1 & 10 & 5.0599e-05\\
		\hline
		  & 100 &8.8863e-07&    & 100& 1.4408e-06\\
		\hline
		 2 & 10 & 3.1229e-05 &  2 & 10 & 5.3995e-05\\
		\hline
		  & 100 & 1.7762e-06 &   & 100& 2.8798e-06\\
		 \hline
		 3 & 10 & 3.1687e-05 &  3 & 10 & 5.5525e-05\\
		 \hline
		  & 100 & 2.6620e-06&   & 100& 4.3158e-06\\
	\end{tabular}
	\caption{Error of the covariance matrix $\mathcal P(T)$ in \eqref{lde} of size $3900 \times 3900$ in Frobenius norm \eqref{norm} obtained by the low-rank integrator described in Algorithm~\ref{algorithm_Lie} with step size $\tau = 0.1$ days. The reference solution is computed with step size $\tau = 0.001$ days.}
\label{tab:nino}
\end{table}

We further make a realization of the stochastic random field. In Figures \ref{fig:pred_dec13} and \ref{fig:pred_dec} the SST anomalies are given in degree Celsius. We show a section of the Indo-Pacific ocean with Australia being on the lower left part and America on the right part. The black part indicates the land, bright colors indicate higher temperatures and dark ones lower temperatures than usual. 
In Figure \ref{fig:pred_dec13} one can see that the temperature is nearly uniformly distributed, whereas we observe from Figure \ref{fig:pred_dec} the typical face of an El Ni\~no event with the unusual warming of the sea surface.

\begin{figure}[h]
	\centering
	\includegraphics[width=\columnwidth]{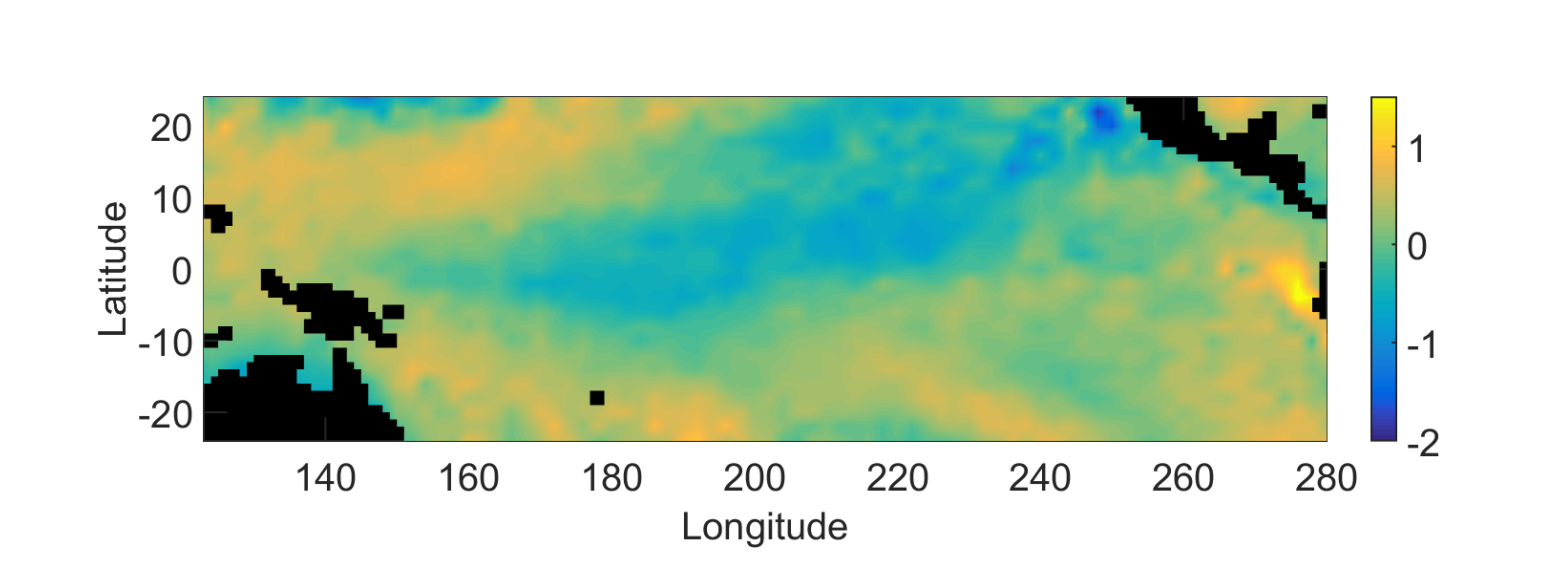}
	\caption{Simulation of the SST anomalies for December $2013$ obtained by the low-rank integrator described in Algorithm~\ref{algorithm_Lie} with rank 10 and using $3900$ grid points.}
	\label{fig:pred_dec13}
\end{figure}
\begin{figure}[h]
	\centering
	\includegraphics[width=\columnwidth]{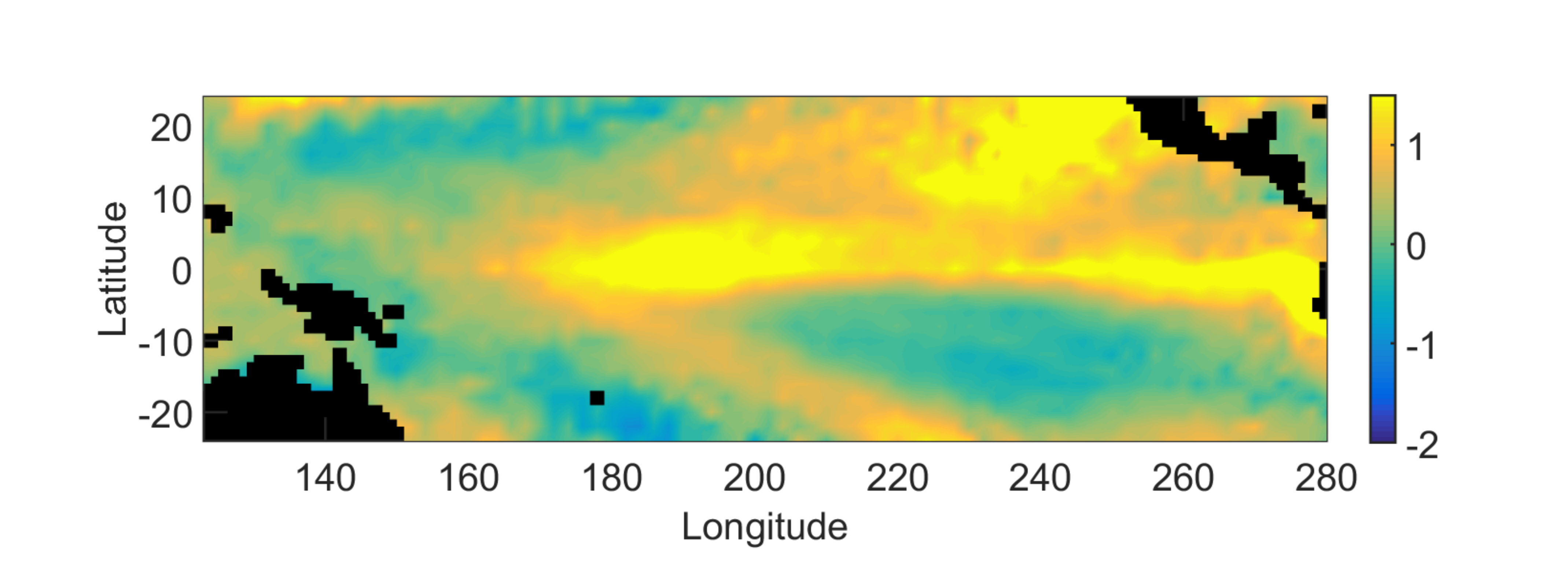}
	\caption{Simulation of the SST anomalies for December $2015$ obtained by the low-rank integrator described in Algorithm~\ref{algorithm_Lie} with rank 10 and using $3900$ grid points.}
	\label{fig:pred_dec}
\end{figure}


\section{Differential Riccati equations}\label{sect:DRE}
We now consider the case where $G(t,X) = Q-XPX$ with constant matrices $Q$ and $P$. The resulting differential equation
\begin{equation} \label{DRE}
\begin{aligned}
\dot{X}(t) & = AX(t) + X(t)A\trasp + Q - X(t) P X(t), \\
X(t_0) & = X_0
\end{aligned}
\end{equation}
is called differential Riccati  equation. Again $X(t)$, $A$, $Q$, $P$ $\in \mathbb{R}^{d \times d}$, and \hbox{$t\in[t_0,T]$}.
The matrices $Q$ and $P$, and the initial value $X_0$ are symmetric and positive semidefinite.
The global existence and positive semidefiniteness of the solution is guaranteed under these conditions, see \cite{DE94}.

\subsection{A low-rank split-step integrator}
Following the general idea of the paper, we split \eqref{DRE} into the following two subproblems:
\begin{subequations}
	\begin{align}
	\label{sub1_ricc} \dot{M}(t) = AM(t)+M(t)A\trasp, \quad & M(t_0) = M_0, \\
	\label{sub2_ricc} \dot{N}(t) = Q-N(t)PN(t), \quad & N(t_0)= N_0, \end{align}
\end{subequations}
where $t \in [t_0,T]$.
The linear problem \eqref{sub1_ricc} is handled as in Section \ref{integrator_lyap}.
The projector-splitting integrator for the solution of \eqref{sub2_ricc} can be specified taking into account the form of the non-linearity. Moreover, we propose the same kind of modification as for DLEs in order to get a low-rank decomposition of the form $USU\trasp$.
The resulting algorithm for the Lie splitting is given in Algorithm \ref{algorithm_Lie_Ricc}.

\begin{algorithm}
	\caption{The first-order low-rank split-step integrator for DREs}
	\label{algorithm_Lie_Ricc}
	\begin{algorithmic}[1]
		\State As in Algorithm \ref{algorithm_Lie}.
		\State As in Algorithm \ref{algorithm_Lie}.
		\State As in Algorithm \ref{algorithm_Lie}.
		\State As in Algorithm \ref{algorithm_Lie}.
		\State Given $U_n^A$ and $S_n^A$ perform one integration step:\label{step5bis}
		\begin{enumerate}
			\item[a.]  Solve $\dot{K}(t) = QU_n^A - K(t)(U_n^A)\trasp P K(t)$ with initial value $K(t_n)= U_n^A S_n^A$.
			Then orthogonalize $K(t_{n+1})$ by a QR decomposition and set $U_{n+1}\widehat{S}_{n+1} = K(t_{n+1})$, where $U_{n+1} \in \mathbb{R}^{d \times r}$ has orthonormal columns and $\widehat{S}_{n+1}\in \mathbb{R}^{r \times r}$.
			\item[b.] Solve $\dot{S}(t) = -U_{n+1}\trasp QU_n^A + S(t)(U_n^A)\trasp P U_{n+1} S(t)$ with initial value $S(t_n)= \widehat{S}_{n+1}$.
			Then set $\widetilde{S}_n = S(t_{n+1})$.
			\item[c.] Solve $\dot{L}(t) = U_{n+1}\trasp Q - L(t) P U_{n+1} L(t)$ with initial value $L(t_n)=\widetilde{S}_n(U_n^A)\trasp$. Then set $S_{n+1} = L(t_{n+1})U_{n+1}$.
		\end{enumerate}
		\State As in Algorithm \ref{algorithm_Lie}.
		\State As in Algorithm \ref{algorithm_Lie}.
	\end{algorithmic}
\end{algorithm}

\subsection{Numerical results for an optimal control problem}
DREs arise, e.g., in optimal control for linear quadratic regulator problems with finite time horizon $T$ for parabolic partial differential equations. Thus we consider the linear control system
\[\dot{x} = Ax + Bu, \quad x(t_0) = x_0, \]
where $A \in \mathbb{R}^{d \times d}$ and $B \in \mathbb{R}^{d \times m}$ are the system matrices, $x \in \mathbb{R}^d$ are the state variables and $u \in \mathbb{R}^m$ is the control.
The output $y \in \mathbb{R}^{q}$ is defined as $y = Cx$, where $C \in \mathbb{R}^{q \times d}$. Both $m$ and $q$ are much smaller than the number $d$ of degrees of freedom.
The functional that has to be minimized is
\[ \mathcal{J}(u,x) = \frac{1}{2}\int_{t_0}^{T} \Big(x(t)\trasp C\trasp Q C x(t)+u(t)\trasp R u(t)\Big) \dd t,
\]
where $Q \in \mathbb{R}^{q \times q}$ is symmetric and positive semidefinite, and $R \in \mathbb{R}^{m \times m}$ is symmetric and positive definite.
The optimal control is given in feedback form by $ u_{\text{opt}} (t) = -R^{-1}B\trasp X(t)x(t)$,  where $X(t)$ is the solution of the following DRE
\begin{equation} \label{rde_control}
\dot{X}(t) = A\trasp X(t) + X(t)A + C\trasp Q C - X(t) BR^{-1}B\trasp X(t). \end{equation}
Note that the solution of \eqref{rde_control} converges for $T \rightarrow \infty$ to a steady state. This limit is given as the solution of the algebraic equation
\begin{equation}
\label{alg_rde} 
0 = A\trasp X(t) + X(t)A + C\trasp Q C - X(t) BR^{-1}B\trasp X(t).
\end{equation}

In order to illustrate the behaviour of Algorithm \ref{algorithm_Lie_Ricc}, we consider a test example proposed in \cite{P}. We consider the following diffusion-advection equation
\begin{equation} 
\label{diff-ad}
\partial_t w  = \Delta w - 10x \partial_x w - 100y \partial_y w,   \qquad  w|_{\partial \Omega} = 0 
\end{equation}
on $\Omega = (0,1)^2$ with homogeneous Dirichlet boundary conditions.
The matrix $A$ arises from the spatial discretization of \eqref{diff-ad} using standard central finite differences, with $\widetilde{d}$ uniformly spaced grid points in each dimension.
We denote the discretization points in the interior of $\Omega$ in $x$ direction with $x_i = i \delta$, $\delta = (\widetilde{d}+1)^{-1}$ for $i=1,\dots, \widetilde{d}$. We take $B \in \mathbb{R}^{d \times 1}$ with
\begin{align*}
B_i =
\begin{cases}
1 \quad \text{if} \quad 0.1<x_i \leq 0.3,\\
0 \quad \text{else},
\end{cases}
\end{align*}
and $C \in \mathbb{R}^{1\times d}$ with
\begin{align*}
C_i =
\begin{cases}
1 \quad \text{if} \quad 0.7<x_i\leq 0.9, \\
0 \quad \text{else}.
\end{cases}
\end{align*}
Further, we choose $R = I$ and $Q = 100I$.

\begin{figure}[tbp]
	\begin{minipage}[t]{.5\textwidth}
		\centering
		\includegraphics[width = \columnwidth]{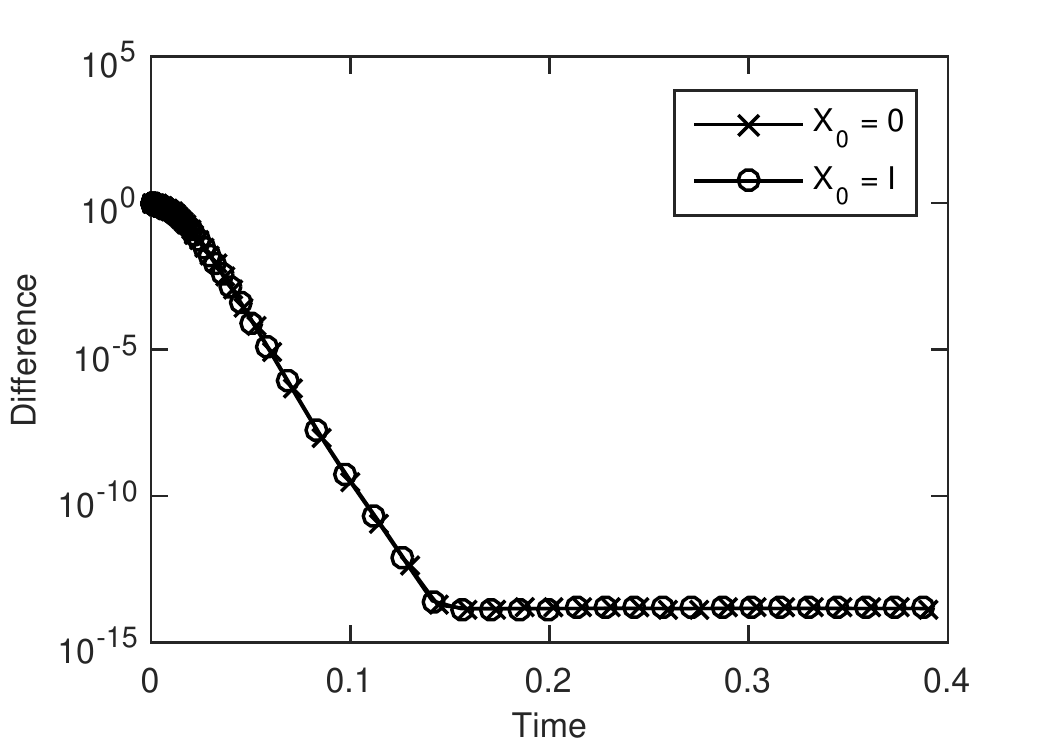}
	\end{minipage} \hfill
	\begin{minipage}[t]{.5\textwidth}
		\centering
		\includegraphics[width = \columnwidth]{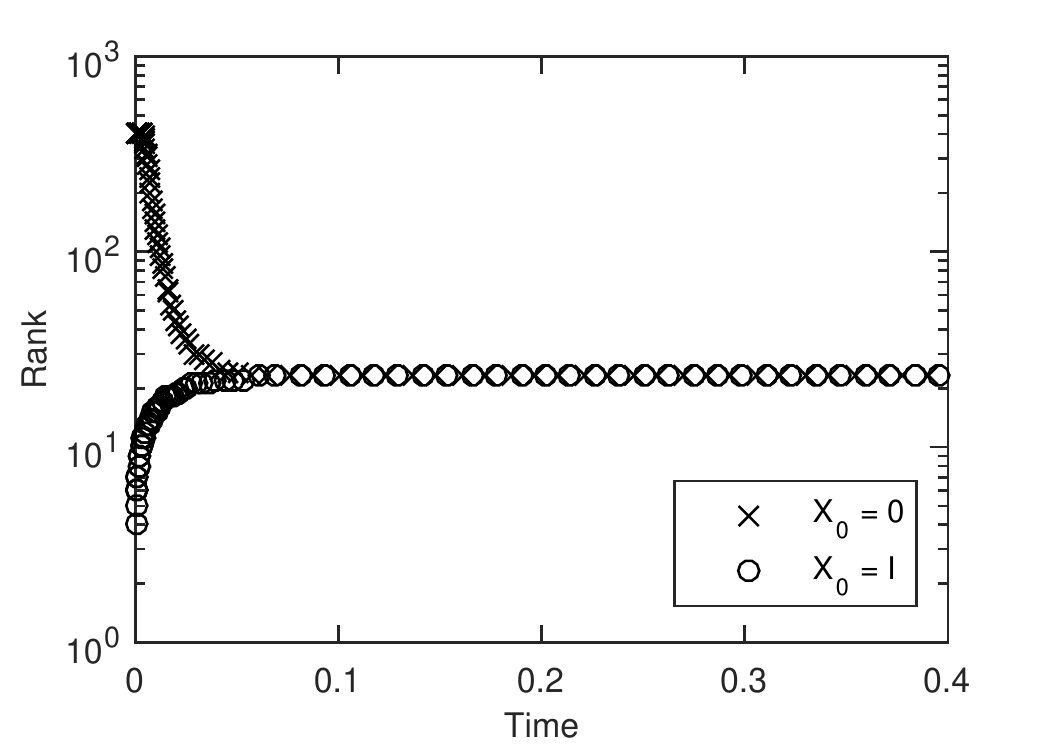}
	\end{minipage}
	\caption{Results for the DRE \eqref{rde_control} for $d=400$ and initial values $X_0 = 0$ and $X_0 = I$, respectively. Left: Difference in Frobenius norm \eqref{norm} between \eqref{alg_rde} and \eqref{rde_control} as a function of time. Right: Rank of the reference solution computed with DOPRI5 as a function of time.}
	\label{fig:ricc_ref}
\end{figure}

In Figure \ref{fig:ricc_ref} we show the behaviour of a reference solution of \eqref{rde_control} computed with DOPRI5.
The left figure shows the convergence of the solution of \eqref{rde_control} to the solution of \eqref{alg_rde} for two different initial values $X_0 = 0$ and $X_0 = I$, respectively. As expected this limit is independent of the choice of the initial data. The relative difference between \eqref{rde_control} and \eqref{alg_rde} in Frobenius norm is shown as function of time.
Further, in Figure \ref{fig:ricc_ref}, right we display the rank of the reference solution as a function of time.

In the following numerical example we take the initial value $X_0 = 0$, the final time $T = 0.1$ and the above reference solution computed with high accuracy. We show the results for $d=400$.
The equations in step~\ref{step5bis} of Algorithm \ref{algorithm_Lie_Ricc} are quadratic matrix differential equations.
We solve them by means of the classical explicit Runge--Kutta method of order 4 with the same step size.

\begin{figure}[btp]
	\begin{minipage}[t]{.49\textwidth}
		\centering
		\includegraphics[width = \columnwidth]{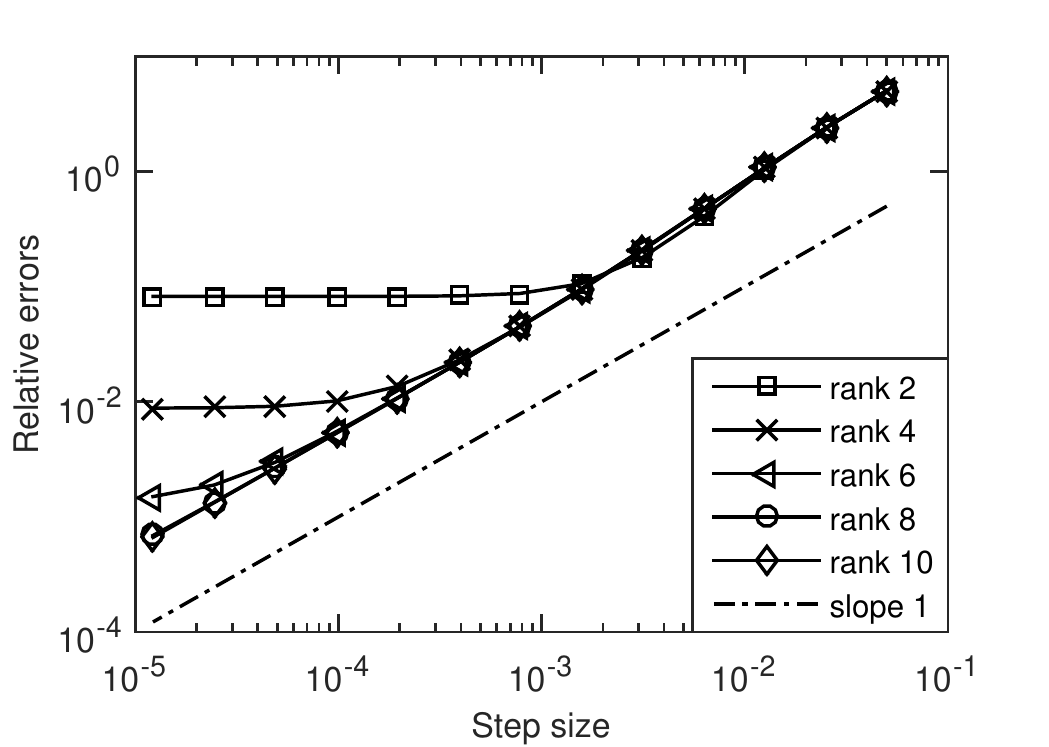}
	\end{minipage} \hfill
	\begin{minipage}[t]{.49\textwidth}
		\centering
		\includegraphics[width = \columnwidth]{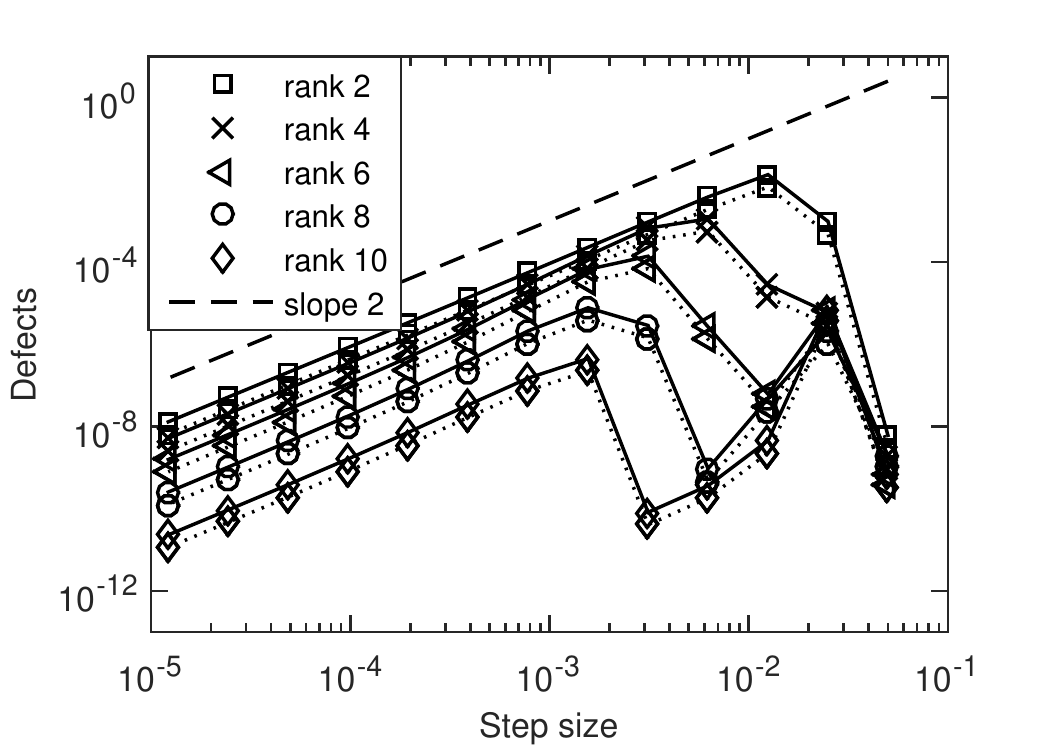}
	\end{minipage}
	\caption{Results for the DRE \eqref{rde_control} for $d=400$. Left: Errors of Lie splitting described in Algorithm \ref{algorithm_Lie_Ricc} in the Frobenius norm \eqref{norm} as a function of step size and rank at $T=0.1$. Right: Defect in symmetry \eqref{def_sym} (continuous line) and in positive semidefiniteness \eqref{def_psd} (dotted line) for Algorithm \ref{algorithm_Lie_Ricc} for different ranks as a function of step size at $T=0.1$.}
	\label{fig:ricc_lie}
\end{figure}

In Figure~\ref{fig:ricc_lie}, left we show the error behaviour of the Lie splitting given in Algorithm \ref{algorithm_Lie_Ricc}. As pointed out above for DLEs, we observe that the error is composed by two different contributions. The choice of a small approximation rank results in stagnation of the error around certain rank-dependent values. On the other hand, if the approximation error becomes small enough, one observes the usual order of convergence one for the outer Lie splitting. In Figure~\ref{fig:ricc_strang}, left we present the corresponding results for Strang splitting. It shows the expected order of convergence two for sufficiently high rank.

In Figures~\ref{fig:ricc_lie}, right and \ref{fig:ricc_strang}, right we show the defects in symmetry \eqref{def_sym} and positive semidefiniteness \eqref{def_psd} in continuous and dotted lines, respectively.
Although our method does not preserve these two features of the solutions, it is remarkable that the defects are always negligible with respect to the overall error of the method.

\begin{figure}[tbp]
	\begin{minipage}[t]{.49\textwidth}
		\centering
		\includegraphics[width = \columnwidth]{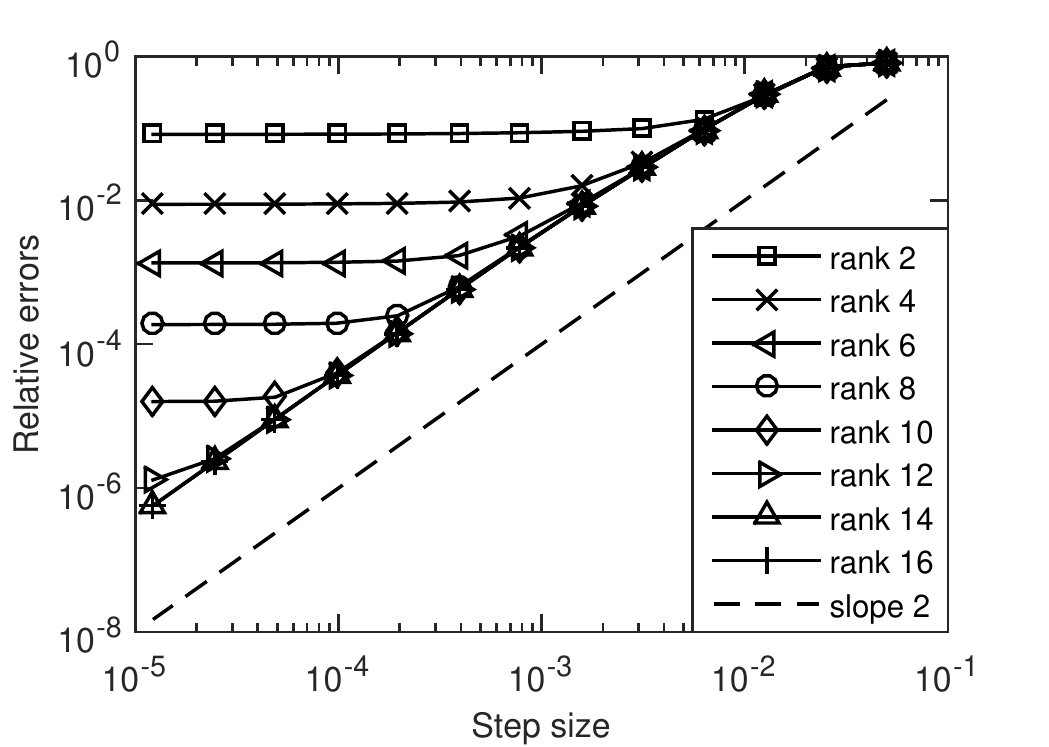}
	\end{minipage}
	\begin{minipage}[t]{.49\textwidth}
		\centering
		\includegraphics[width = \columnwidth]{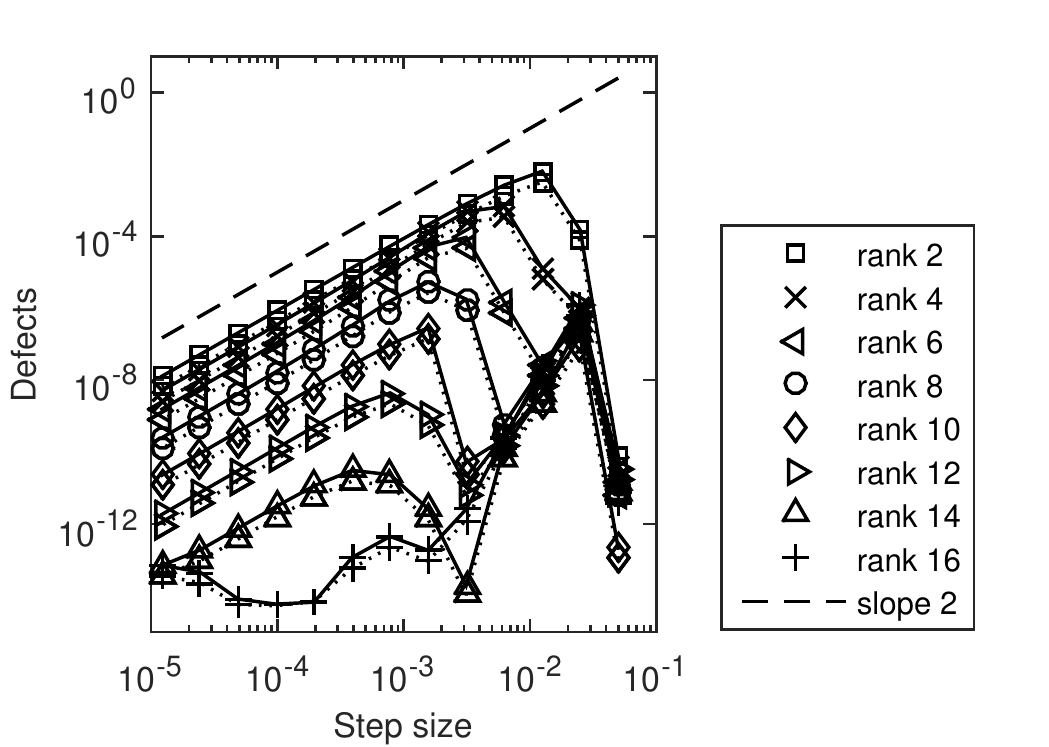}
	\end{minipage}
	\caption{Results for the DRE \eqref{rde_control} for $d=400$. Left: Errors of Strang splitting in Frobenius norm \eqref{norm} as a function of step size and rank at $T=0.1$. Right: Defect in symmetry \eqref{def_sym} (continuous line) and in positive semidefiniteness \eqref{def_psd} (dotted line) for Strang splitting for different ranks as a function of step size at $T=0.1$.}
	\label{fig:ricc_strang}
\end{figure}

\subsection{Generalized differential Riccati equations}

In a stochastic version of the linear quadratic regulator problem, a generalised differential Riccati equation (GDRE) arises
\begin{equation}\label{SDRE}
\begin{aligned}
\dot X(t) &= A\trasp X(t) + X(t)A + Q + C\trasp X(t)C - X(t)B R^{-1} B\trasp X(t) , \quad t \in [t_0,T], \\
X(0) &= H.
\end{aligned}
\end{equation}
We point out that \eqref{SDRE} has the same structure as the DRE \eqref{DRE} except for the term $C\trasp XC$. The matrix $C$ is a weighting term in the noise perturbation, see \cite{YongZ99}. Recently, a numerical method for GRDEs of the form \eqref{SDRE} was proposed in \cite{NLA:NLA2091}.

The above GDRE shares the structure of \eqref{eq} with $G = Q + C\trasp XC - XB R^{-1} B\trasp X$.
It is therefore possible to employ the approach explained in Section \ref{integrator}. Two subproblems of the form  \eqref{sub1} and \eqref{sub2} arise and they can be solved in low-rank form as presented in this work. The projector-splitting integrator can be adapted to the form of the nonlinearity, however, a modification to preserve symmetry and positive semidefiniteness might be required.

\section{Conclusions}
We proposed a new low-rank integrator for a class of matrix differential equations which includes, among others, differential Lyapunov and differential Riccati equations.
A low-rank approximation of the solution is computed in a dynamical way, working only with the low-rank factors of the solution. This approach gives substantial advantages in terms of computing time and memory requirements.
Moreover, the integrator can handle stiffness in an efficient way. Splitting methods form the core ingredient of the new method. They make it possible to treat the stiff part of the equation separately from the non-stiff one.
Numerical results for differential Lyapunov and differential Riccati equations are discussed, and a simulation of the weather phenomenon El Ni\~no is presented.

\section*{Acknowledgements}
This work was supported by the Austrian Science Fund (FWF) -- project id: P27926.

\bibliographystyle{plain}
\section*{\refname}
\bibliography{MOPP17}

\end{document}